\documentclass[10pt, reqno]{amsart}
\usepackage{amsmath, amsthm, amscd, amsfonts, amssymb, graphicx, color}
\usepackage[bookmarksnumbered, colorlinks, plainpages]{hyperref}
\hypersetup{colorlinks=true,linkcolor=red, anchorcolor=green, citecolor=cyan, urlcolor=red, filecolor=magenta, pdftoolbar=true}
\usepackage{mathrsfs}
\usepackage[OT2,OT1]{fontenc}
\newcommand\cyr
{
\renewcommand\rmdefault{wncyr}
\renewcommand\sfdefault{wncyss}
\renewcommand\encodingdefault{OT2}
\normalfont
\selectfont
}
\DeclareTextFontCommand{\textcyr}{\cyr}

\newtheorem{theorem}{Theorem}[section]
\newtheorem{lemma}[theorem]{Lemma}

\newtheorem{corollary}[theorem]{Corollary}
\theoremstyle{definition}

\theoremstyle{remark}
\newtheorem{remark}[theorem]{Remark}
\numberwithin{equation}{section}

\begin{document}
\setcounter{page}{1}

\title[class number one conjecture]{Lambert $W$-function %Fixed points of quantum functor 
and Gauss class number one conjecture}

\author[Nikolaev]
{Igor V. Nikolaev$^1$}

\address{$^{1}$ Department of Mathematics and Computer Science, St.~John's University, 8000 Utopia Parkway,  
New York,  NY 11439, United States.}
\email{\textcolor[rgb]{0.00,0.00,0.84}{igor.v.nikolaev@gmail.com}}

\dedicatory{All data are available as part of the manuscript}
%\dedicatory{In memory of Ola Bratteli}



%\date{Received:  August 14, 2015; Revised: yyyyyy; Accepted: zzzzzz.}

\begin{abstract}
We study fixed points of a function arising in a representation theory 
of the Drinfeld modules by the bounded linear operators on a Hilbert space. 
We prove that such points correspond to number fields of the class number one.
As an application, one gets a solution to  the Gauss conjecture 
for the real quadratic fields of class number one.

\end{abstract}

\maketitle

%**************************************************************************
\section{Introduction}
%***************************************************************************
Drinfeld modules are powerful invariants of the non-abelian class field theory for the function fields
[Drinfeld 1974] \cite{Dri1}. 
Recall that if  $\mathfrak{k}:=\mathbf{F}_{p^n}$ is  a finite field and $\tau(x)=x^p$, then one can  consider a ring $\mathfrak{k}\langle\tau\rangle$
of the non-commutative polynomials given by the commutation relation $\tau a=a^p\tau$ for all $a\in A$, where $A:=\mathfrak{k}[T]$
is the ring of polynomials in  variable $T$ over $\mathfrak{k}$.  
The Drinfeld module $Drin_A^r(\mathfrak{k})$ of rank $r\ge 1$ is a homomorphism 
$\rho: A\buildrel r\over\longrightarrow \mathfrak{k}\langle\tau\rangle$
given by a polynomial $\rho_a=a+c_1\tau+\dots+c_r\tau^r$,
where $a\in A$ and $c_i\in \mathfrak{k}$ 
[Rosen 2002] \cite[Section 12]{R}. 
Consider a torsion submodule $\Lambda_{\rho}[a]:=\{\lambda\in\overline{\mathfrak{k}} ~|~\rho_a(\lambda)=0\}$ of
the $A$-module $\overline{\mathfrak{k}}$. 
 Drinfeld modules $Drin_A^r(\mathfrak{k})$ and associated torsion submodules $\Lambda_{\rho}[a]$ 
 define generators of a non-abelian class field theory for the function fields. 
 Namely, for each non-zero $a\in A$ the function 
field $\mathfrak{k}(T)\left(\Lambda_{\rho}[a]\right)$  is a Galois extension of the field  $\mathfrak{k}(T)$
of rational functions in variable $T$ over $\mathfrak{k}$,
such that its  Galois group is isomorphic to a subgroup of the matrix group $GL_r\left(A/aA\right)$
[Rosen 2002] \cite[Proposition 12.5]{R}. 

We consider the norm closure of a representation of the multiplicative semi-group  [Li 2017] \cite{Li1} of the ring 
$\mathfrak{k}\langle\tau\rangle$  by  bounded linear operators on a Hilbert space \cite{Nik1}.
The latter is   a $C^*$-algebra $\mathscr{A}_{RM}^{2r}$  generated by the unitary operators $u_1,\dots, u_{2r}$ 
satisfying the commutation relations $\{u_ju_i=e^{2\pi i\theta_{ij}}u_iu_j ~|~1\le i,j\le 2r\}$,
where $\theta_{ij}$ are algebraic numbers and  $\Theta=(\theta_{ij})\in M_{2r}(\mathbf{R})$ is a 
skew-symmetric matrix  [Rieffel 1990] \cite{Rie1}.
The Grothendieck semi-group  [Blackadar 1986] \cite[Chapter III]{B}  of $\mathscr{A}_{RM}^{2r}$ is given by the formula  $K_0^+(\mathscr{A}_{RM}^{2r})\cong \mathbf{Z}+\alpha_1\mathbf{Z}+\dots+
\alpha_{r}\mathbf{Z}\subset \mathbf{R}$, where $\alpha_j\in\mathbf{R}$ are algebraic integers of degree $2r$ over $\mathbf{Q}$.  
 The following is true  \cite[Theorem 3.3]{Nik1}:  (i)  there exists a functor $F: Drin_A^{r}(\mathfrak{k})\mapsto \mathscr{A}_{RM}^{2r}$
from the category of Drinfeld  modules $\mathfrak{D}$ to a category 
of the noncommutative tori $\mathfrak{A}$,   which maps any pair of isogenous  (isomorphic, resp.) 
modules  $Drin_A^{r}(\mathfrak{k}), ~\widetilde{Drin}_A^{r}(\mathfrak{k})\in \mathfrak{D}$
to a pair of the homomorphic (isomorphic, resp.)  tori  $\mathscr{A}_{RM}^{2r}, \widetilde{\mathscr{A}}_{RM}^{2r}
\in \mathfrak{A}$,  (ii) 
$F(\Lambda_{\rho}[a])=\{e^{2\pi i\alpha_j+\log\log\varepsilon} ~|~1\le j\le r\}$,
where $\mathscr{A}_{RM}^{2r}=F(Drin_A^r(\mathfrak{k}))$ and   $\varepsilon$ is a unit of the number field $\mathbf{Q}(\alpha_j)$ 
and (iii) the number field $K=\mathbf{Q}(F(\Lambda_{\rho}[a]))$ is the extension of its subfield 
with the Galois group $G\subseteq GL_r\left(A/aA\right)$.  
 The above formulas imply a non-abelian class field theory for the number fields. Namely, 
fix a non-zero $a\in A$ and let $G:=Gal~(\mathfrak{k}(\Lambda_{\rho}[a]) ~|~ \mathfrak{k})\subseteq GL_r(A/aA)$,
where  $\Lambda_{\rho}[a]$ is the torsion submodule of the $A$-module  $\overline{\mathfrak{k}_{\rho}}$.
Consider the number field $K=\mathbf{Q}(F(\Lambda_{\rho}[a]))$. 
If $k$ is the maximal subfield of $K$ fixed by the action of 
all elements of the group $G$, then
the number field
%***********************************************
\begin{equation}\label{eq1.1}
K\cong
\begin{cases} k\left(e^{2\pi i\alpha_j +\log\log\varepsilon}\right), & if ~k\subset(\mathbf{C} - \mathbf{R})\cup\mathbf{Q},\cr
               k\left(\cos 2\pi\alpha_j \times\log\varepsilon\right), & if ~k\subset\mathbf{R}
\end{cases}               
\end{equation}
%***************************************************
is a Galois extension of  $k$,
 such that  $Gal~(K |k)\cong G$ \cite[Corollary 3.4]{Nik1}. 
Specifically,
$k\cong\mathbf{Q}(i\alpha_j)$  if $k\subset(\mathbf{C} - \mathbf{R})\cup\mathbf{Q}$ and 
        $k\cong\mathbf{Q}(\alpha_j)$ if $k\subset\mathbf{R}$ (Lemma \ref{lm3.1}).

\bigskip
The aim of our note are  number fields $k\subseteq K$, such that $K\cong k$.  
It follows from (\ref{eq1.1})  that this  property depends on the solvability of equations
in variables $\{\alpha_j\in \mathbf{R} ~|~ 1\le j\le r\}$:
%***********************************************
\begin{equation}\label{eq1.2}
\begin{cases} i\alpha_j=e^{2\pi i\alpha_j +\log\log\varepsilon}, & if ~k\subset(\mathbf{C} - \mathbf{R})\cup\mathbf{Q},\cr
               \alpha_j=\cos 2\pi\alpha_j \times\log\varepsilon, & if ~k\subset\mathbf{R}, 
\end{cases}               
\end{equation}
%***************************************************
where $\varepsilon\in O_k^{\times}$ is a constant  in the group  of units  $O_k^{\times}$ of the field $k$.
Denote by $W_j(z)$ the $j$-th branch of the Lambert $W$-function [Corless, Gonnet, Hare, Jeffrey \&  Knuth 1996] \cite{CGHJK}. 
Our main results  can be formulated as follows.
%**************************************
\begin{theorem}\label{thm1.1}
The number fields $K\cong k$ given by formulas (\ref{eq1.1}) are isomorphic,  if and only if: 
%***********************************************
\begin{equation}\label{eq1.3}
\begin{cases} \alpha_j=-\frac{1}{2\pi i}  ~W_j(-2\pi\log \varepsilon), & if ~\varepsilon\in(\mathbf{C} - \mathbf{R})\cup\mathbf{Q},\cr
%&\cr
 \alpha_j=-\frac{1}{2\pi i}  ~\left[W_j(-2\pi i\log \varepsilon)- W_j(2\pi i\log \varepsilon)\right], & if ~\varepsilon\in\mathbf{R}.
 \end{cases}               
\end{equation}
%***************************************************
\end{theorem}
%*******************************************

\medskip
Theorem \ref{thm1.1} can be used to calculate the class number $h_k$ of the field $k$. Indeed, 
consider the Hilbert class field $K$ of the number field $k$. 
By the class field theory, the class group
 $Cl~(k)\cong Gal~(K|k)\subseteq GL_r\left(A/aA\right)$  is trivial if and only if  $h_k:=|Cl~(k)|=1$. 
 In other words,  the set $\{\alpha_j\}$  of roots (\ref{eq1.3})  is counting
  fields $k$ having class number one. 
Namely, one gets the following corollary.
%**************************************
\begin{corollary}\label{cor1.2}
If $k$ is a Galois extension of degree $2r$ over $\mathbf{Q}$, then:
%***********************************************
\begin{equation}\label{eq1.4}
\#\{k~|~h_k=1\}=
\begin{cases} 8, & if ~r= 1 ~and~ k\subset (\mathbf{C}-\mathbf{R})\cup\mathbf{Q}, \cr
                \infty, & if ~r= 1 ~and~ k\subset \mathbf{R},       \cr
                \infty, & if ~r\ge 2. 
\end{cases}               
\end{equation}
%***************************************************
\end{corollary}
%*******************************************
%*********************************************************
\begin{remark}\label{rmk1.3}
The total number of imaginary quadratic fields $\mathbf{Q}(\sqrt{d})$ 
of class number one is known to be $9$ corresponding to the discriminants 
$d\in \{-1, -2, -3, -7, -11,$\linebreak
$ -19, -43, -67, -163\}$.  
However, the value $d=-1$ cannot be covered by our method. Indeed, 
the root $\pm i$ of the minimal polynomial $p(x)=x^2+1$ is linearly independent over $\mathbf{Q}$,
while the root $\pm 1$  of the corresponding minimal polynomial $q(x)=x^2-1$ does not, see item (i) in the proof of Lemma \ref{lm3.1}
for the notation and details. Hence the first line in  formula (\ref{eq1.4}). 
%\end{remark}
%********************************************************
%*********************************************************
%\begin{remark}
The middle line in formula (\ref{eq1.4}) corresponds to a
class number one conjecture for the real quadratic number fields
dating back to 
[Gauss 1801] \cite[Article 304]{G}.
\end{remark}
%********************************************************

\medskip
The paper is organized as follows.  A brief review of the preliminary facts is 
given in Section 2. Theorem \ref{thm1.1} and Corollary \ref{cor1.2} 
are proved in Section 3.

%**************************************************************************
\section{Preliminaries}
%***************************************************************************
We briefly review the  Lambert $W$-function and non-abelian class field theory. 
We refer the reader to [Corless, Gonnet, Hare, Jeffrey \&  Knuth 1996] \cite{CGHJK},   
[Rieffel 1990] \cite{Rie1},  [Rosen 2002] \cite[Chapters 12 \& 13]{R} 
and \cite{Nik1}  for a detailed exposition.

%**************************************************************************
\subsection{Lambert $W$-function}
%***************************************************************************
The Lambert $W$-function is a multivalued inverse of the function:
%***********************************************************************
\begin{equation} \label{eq2.1}
f(w)=we^w, \quad \hbox{where}\quad w\in\mathbf{C}. 
\end{equation}
%******************************************************************
For each $i\in\mathbf{Z}$ there is a branch of the Lambert $W$-function 
denoted by $W_i(z)$. 
In other words, if $z$ and $w$ are any complex numbers, then 
%***********************************************************************
\begin{equation} \label{eq2.2}
z=we^w, 
\end{equation}
%******************************************************************
if and only if $w=W_j(z)$ for some  $j\in\mathbf{Z}$.  
We denote the principal branch by  $W(z):=W_1(z)$.  
The Taylor series of the Lambert $W$-function is given by the formula:
%***********************************************************************
\begin{equation} \label{eq2.3}
W(z)=\sum_{n=1}^{\infty} \frac{(-n)^{n-1}}{n!} z^n. 
\end{equation}
%******************************************************************
The following lemma is an implication of the formula (\ref{eq2.2}).
%***********************************************************
\begin{lemma}\label{lm2.1}
{\bf (\cite[p. 332]{CGHJK})}
Let $A,B$ and $C$ be complex numbers, such that $BC\ne 0$. 
A root of the equation:
%***********************************************************************
\begin{equation} \label{eq2.4}
z=A+Be^{Cz}
\end{equation}
%******************************************************************
is given by the general formula:
%***********************************************************************
\begin{equation} \label{eq2.5}
z=A-\frac{1}{C}W\left(-BCe^{AC}\right). 
\end{equation}
%******************************************************************
\end{lemma}
%***********************************************************

%**************************************************************************
\subsection{Non-abelian class field theory}
%***************************************************************************
Let  $\mathfrak{k}:=\mathbf{F}_q(T)$ ($A:=\mathbf{F}_q[T]$, resp.) be the field of rational functions (the ring of polynomial functions, resp.)
in one variable $T$ over a finite field $\mathbf{F}_q$, where $q=p^n$
and let  $\tau_p(x)=x^p$. 
Recall that  the  Drinfeld module $Drin_A^{r}(\mathfrak{k})$   of rank $r\ge 1$
is a homomorphism
%********************************************************************
\begin{equation}\label{eq2.6}
\rho:  ~A\buildrel r\over\longrightarrow \mathfrak{k}\langle\tau_p\rangle
\end{equation}
%****************************************************************
given by a polynomial $\rho_a=a+c_1\tau_p+c_2\tau_p^2+\dots+c_r\tau_p^r$ with $c_i\in \mathfrak{k}$ and $c_r\ne 0$, 
such that for all $a\in A$ the constant term of $\rho_a$ is $a$ and 
$\rho_a\not\in \mathfrak{k}$ for at least one $a\in A$ [Rosen 2002] \cite[p. 200]{R}.
For each non-zero $a\in A$ the function 
field $\mathfrak{k}\left(\Lambda_{\rho}[a]\right)$  is a Galois extension of $\mathfrak{k}$,
such that its  Galois group is isomorphic to a subgroup $G$ of the matrix group $GL_r\left(A/aA\right)$,
where   $\Lambda_{\rho}[a]=\{\lambda\in\overline{ \mathfrak{k}} ~|~\rho_a(\lambda)=0\}$
is a torsion submodule of the non-trivial  Drinfeld module  $Drin_A^{r}(\mathfrak{k})$  [Rosen 2002] \cite[Proposition 12.5]{R}.
Clearly, the abelian extensions correspond to the case $r=1$.

Let $G$ be a  left cancellative  semigroup generated by $\tau_p$ and all  $a_i\in \mathfrak{k}$ subject to the commutation relations 
$\tau_p a_i=a_i^p\tau_p$.
%*****************************************************************************************************
\footnote{In other words, we omit the additive structure and consider a multiplicative semigroup of the ring $\mathfrak{k}\langle\tau_p\rangle$.  }
%***************************************************************************************************
  Let $C^*(G)$ be the semigroup $C^*$-algebra [Li 2017] \cite{Li1}.  
For a Drinfeld module  $Drin_A^{r}(\mathfrak{k})$  defined  by  (\ref{eq2.6}) we consider a homomorphism of the semigroup $C^*$-algebras:  
 %********************************************************************
\begin{equation}\label{eq2.7}
C^*(A)\buildrel r\over\longrightarrow C^*(\mathfrak{k}\langle\tau_p\rangle). 
\end{equation}
%****************************************************************
It is proved that (\ref{eq2.7}) defines a map  $F: Drin_A^{r}(\mathfrak{k})\mapsto \mathscr{A}_{RM}^{2r}$ \cite[Definition 3.1]{Nik1}. 
%***************************************************************
\begin{theorem}\label{thm2.1}
{\bf (\cite{Nik1})}
The following is true:

\medskip
(i) the map $F: Drin_A^{r}(\mathfrak{k})\mapsto \mathscr{A}_{RM}^{2r}$ is a functor 
from the category of Drinfeld  modules $\mathfrak{D}$ to a category 
of the noncommutative tori $\mathfrak{A}$,   which maps any pair of isogenous  (isomorphic, resp.) 
modules  $Drin_A^{r}(\mathfrak{k}), ~\widetilde{Drin}_A^{r}(\mathfrak{k})\in \mathfrak{D}$
to a pair of the homomorphic (isomorphic, resp.)  tori  $\mathscr{A}_{RM}^{2r}, \widetilde{\mathscr{A}}_{RM}^{2r}
\in \mathfrak{A}$;  

\smallskip
(ii) $F(\Lambda_{\rho}[a])=\{e^{2\pi i\alpha_i+\log\log\varepsilon} ~|~1\le i\le r\}$,
where $\mathscr{A}_{RM}^{2r}=F(Drin_A^{r}(\mathfrak{k}))$, 
$\alpha_i$ are generators of the Grothendieck semi-group $K_0^+(\mathscr{A}_{RM}^{2r})$,  $\log\varepsilon$ is a scaling factor
 and $\Lambda_{\rho}(a)$ is the  torsion submodule of the $A$-module $\overline{\mathfrak{k}_{\rho}}$;

\smallskip
(iii) the Galois group $Gal \left(k(e^{2\pi i\alpha_i+\log\log\varepsilon})  ~| ~k\right)\subseteq GL_{r}\left(A/aA\right)$,
where $k$ is a subfield of the number field $\mathbf{Q}(e^{2\pi i\alpha_i+\log\log\varepsilon})$. 
 \end{theorem}
%***************************************************************
Theorem \ref{thm2.1} implies a non-abelian class field theory as follows.
Fix a non-zero $a\in A$ and let $G:=Gal~(\mathfrak{k}(\Lambda_{\rho}[a]) ~|~ \mathfrak{k})\subseteq GL_r(A/aA)$,
where  $\Lambda_{\rho}[a]$ is the torsion submodule of the $A$-module  $\overline{\mathfrak{k}_{\rho}}$.
Consider the number field $K=\mathbf{Q}(F(\Lambda_{\rho}[a]))$. 
Denote by $k$ the maximal subfield of $K$ which is fixed by the action of 
all elements of the group $G$. 
%***************************************************************
\begin{corollary}\label{cor2.2} 
{\bf (Non-abelian class field theory)} 
The number field
%***********************************************
\begin{equation}\label{eq2.8}
K\cong
\begin{cases} k\left(e^{2\pi i\alpha_j +\log\log\varepsilon}\right), & if ~k\subset(\mathbf{C} - \mathbf{R})\cup\mathbf{Q},\cr
               k\left(\cos 2\pi\alpha_j \times\log\varepsilon\right), & if ~k\subset\mathbf{R},
\end{cases}               
\end{equation}
%***************************************************
is a Galois extension of  $k$,
 such that  $Gal~(K |k)\cong G$.
\end{corollary}
%***************************************************************

%**************************************************************************
\section{Proofs}
%***************************************************************************
%**************************************************************************
\subsection{Proof of Theorem \ref{thm1.1}}
%***************************************************************************
For the sake of clarity, let us outline the main ideas. 
Roughly speaking, the proof of Theorem \ref{thm1.1}
consists of a straightforward calculation of the roots 
of equation (\ref{eq1.2}) using formulas (\ref{eq2.4}) 
and (\ref{eq2.5}) for  the Lambert $W$-function. 
We spit the proof in a series of lemmas. 

%***********************************************
\begin{lemma}\label{lm3.1}
%***********************************************
The number field $k\subset K\cong\mathbf{Q}(F(\Lambda_{\rho}[a]))$
is defined by the formulas:
%**************************************************************
\begin{equation}\label{eq3.1}
k\cong \begin{cases} \mathbf{Q}(i\alpha_1, \dots, i \alpha_r), & if ~k\subset(\mathbf{C} - \mathbf{R})\cup\mathbf{Q},\cr
               \mathbf{Q}(\alpha_1, \dots, \alpha_r), & if ~k\subset\mathbf{R}.
\end{cases}               
\end{equation}
%***************************************************
\end{lemma} 
%***********************************************
\begin{proof} (i) {\bf Case $k\subset(\mathbf{C} - \mathbf{R})\cup\mathbf{Q}$.}
Let $k_0$ be a totally real field
and denote by $k_{CM}$ ($k_{RM}$, resp.) its complex (real, resp.) multiplication 
field,  i.e. a totally imaginary (totally real, resp.)  quadratic extension of $k_0$. 
The minimal polynomial of  $k_{CM}$ is an alternating sum of monomials
of the minimal polynomial of the field $k_{RM}$.  Indeed, since both $k_{RM}$
 and $k_{CM}$ are quadratic extensions, their minimal polynomials can be written
 in the form
$p(x)=x^{2r}+a_2x^{2r-2}+\dots+a_{2r-2}x^2+a_{2r}$. 
In particular, if $x$ is a real root of $p(x)$,  then $ix$ is a complex root 
of the polynomial $q(x)=x^{2r}-a_2x^{2r-2}+\dots+(-1)^{r+1}a_{2r-2}x^2+(-1)^ra_{2r}$,
i.e. an alternating sum of the monomials of $p(x)$. 
(We refer the reader to \cite[Remark 6.6.3]{N} for the explicit matrix formulas.)
On the other hand, $Gal~(K ~| ~k_{CM})\cong G\subseteq  GL_{r}\left(A/aA\right)$
and therefore $k\cong k_{CM}\cong \mathbf{Q}(i\alpha_j)$, where $1\le j\le r$.

\medskip 
(ii)   {\bf Case $k\subset\mathbf{R}$.}
Using notation and argument of item (i), 
one gets  \linebreak
$Gal~(K ~| ~k_{RM})\cong G\subseteq  GL_{r}\left(A/aA\right)$.
In particular,  $k\cong k_{RM}\cong \mathbf{Q}(\alpha_j)$,  where $1\le j\le r$.

\bigskip
Lemma \ref{lm3.1} is proved.   
\end{proof}
%*********************************************

%***********************************************
\begin{lemma}\label{lm3.2}
The number fields defined by formulas (\ref{eq2.8}) are isomorphic $K\cong k$ 
 if and only if:  
%***********************************************
\begin{equation}\label{eq3.2}
\begin{cases} i\alpha_j=e^{2\pi i\alpha_j +\log\log\varepsilon}, & if ~k\subset(\mathbf{C} - \mathbf{R})\cup\mathbf{Q},\cr
               \alpha_j=\cos 2\pi\alpha_j \times\log\varepsilon, & if ~k\subset\mathbf{R}.
\end{cases}               
\end{equation}
%***************************************************
\end{lemma} 
%***********************************************
\begin{proof}
In view of Lemma \ref{lm3.1}, one can write (\ref{eq2.8}) in the form:
%***********************************************
\begin{equation}\label{eq3.3}
K\cong
\begin{cases} \mathbf{Q}\left(i\alpha_j, ~e^{2\pi i\alpha_j +\log\log\varepsilon}\right), & if ~k\subset(\mathbf{C} - \mathbf{R})\cup\mathbf{Q},\cr
               \mathbf{Q}\left(\alpha_j, ~\cos 2\pi\alpha_j \times\log\varepsilon\right), & if ~k\subset\mathbf{R},
\end{cases}               
\end{equation}
%***************************************************
where $1\le j\le r$. 

\bigskip
(i) If conditions (\ref{eq3.2}) hold, then formulas (\ref{eq3.3}) imply:
%***********************************************
\begin{equation}\label{eq3.4}
K\cong
\begin{cases} \mathbf{Q}\left(i\alpha_j, ~i\alpha_j\right)\cong k, & if ~k\subset(\mathbf{C} - \mathbf{R})\cup\mathbf{Q},\cr
               \mathbf{Q}\left(\alpha_j, ~\alpha_j\right)\cong k, & if ~k\subset\mathbf{R}. 
\end{cases}               
\end{equation}
%***************************************************
In other words, one gets from (\ref{eq3.4}) an isomorphism of the number fields $K\cong k$.

\medskip
(ii) Conversely, let $K\cong k$. In view of formulas (\ref{eq2.8}), we obtain the following inclusions:
%***********************************************
\begin{equation}\label{eq3.5}
\begin{cases} e^{2\pi i\alpha_j +\log\log\varepsilon}\in k\cong\mathbf{Q}(i\alpha_j), & if ~k\subset(\mathbf{C} - \mathbf{R})\cup\mathbf{Q},\cr
              \cos 2\pi\alpha_j \times\log\varepsilon\in k\cong\mathbf{Q}(\alpha_j), & if ~k\subset\mathbf{R}. 
\end{cases}               
\end{equation}
%***************************************************
Since algebraic numbers $\{e^{2\pi i\alpha_j +\log\log\varepsilon} ~|~1\le j\le r\}$ ($\{ \cos 2\pi\alpha_j \times\log\varepsilon ~|~1\le j\le r\}$, resp.)
are linearly indepenedent over $\mathbf{Q}$, one can take them for a basis of the ring of integers of the field   $k\cong\mathbf{Q}(i\alpha_j)$
($k\cong\mathbf{Q}(\alpha_j)$, resp.). But any such a basis must coincide with $\{i\alpha_j ~|~1\le j\le r\}$  ($\{\alpha_j ~|~1\le j\le r\}$, resp.)
after a linear transformation over $\mathbf{Q}$.  
Thus $e^{2\pi i\alpha_j +\log\log\varepsilon}=i\alpha_j$ ($\cos 2\pi\alpha_j \times\log\varepsilon=\alpha_j$, resp.). 
In other words, one gets the system of equations (\ref{eq3.2}). 

\bigskip
Lemma \ref{lm3.2} is proved. 
\end{proof}
%*********************************************

%***********************************************
\begin{lemma}\label{lm3.3}
The roots $\{\alpha_j ~|~1\le j\le r\}$ of equations (\ref{eq3.2}) are given by the formulas:
%***********************************************
\begin{equation}\label{eq3.6}
\begin{cases} \alpha_j=-\frac{1}{2\pi i}  ~W_j(-2\pi\log \varepsilon), & if ~\varepsilon\in(\mathbf{C} - \mathbf{R})\cup\mathbf{Q},\cr
                \alpha_j=-\frac{1}{2\pi i}  ~\left[W_j(-2\pi i ~\log \varepsilon)- W_j(2\pi i ~\log \varepsilon)\right], & if ~\varepsilon\in\mathbf{R},
\end{cases}               
\end{equation}
%***************************************************
where $W_j(z)$  is a $j$-th branch  of the Lambert $W$-function.  
\end{lemma} 
%***********************************************
\begin{proof}
 (i) {\bf Case $\varepsilon\in(\mathbf{C} - \mathbf{R})\cup\mathbf{Q}$.}
 In this case $\varepsilon\in k\subset(\mathbf{C} - \mathbf{R})\cup\mathbf{Q}$
 and we must use the first equation (\ref{eq3.2}). The latter can be written 
 in an equivalent form:
 %************************************************************
 \begin{equation}\label{eq3.7}
 \beta_j+i\alpha_j=\beta_j+\frac{\log\varepsilon}{e^{2\pi\beta_j}}e^{2\pi(\beta_j+i\alpha_j)},
 \end{equation}
 %******************************************************* 
 where $\beta_j\in\mathbf{R}$ is an arbitrary constant.  We shall denote:
%***********************************************
\begin{equation}\label{eq3.8}
\left\{
\begin{array}{cl}
z= &\beta_j+i\alpha_j, \\
A= & \beta_j,\\
B= & \frac{\log\varepsilon}{e^{2\pi\beta_j}}, \\
C= & 2\pi .
\end{array}
\right.
\end{equation}
%*****************************************************************

 In this notation equation (\ref{eq3.7}) takes the form $z=A+Be^{Cz}$, 
 where $BC\ne 0$.  We apply Lemma \ref{lm2.1} to find the roots of the latter using formula 
 (\ref{eq2.5}) and we get:             
%***********************************************************************
\begin{equation} \label{eq3.9}
\beta_j+i\alpha_j=\beta_j-\frac{1}{2\pi}W_j\left(-2\pi\log\varepsilon\right). 
\end{equation}
%******************************************************************

One can cancel $\beta_j$ in the both sides of equation (\ref{eq3.9}) to  obtain a formula:
%***********************************************************************
\begin{equation} \label{eq3.10}
 \alpha_j=-\frac{1}{2\pi i}  ~W_j\left(-2\pi\log \varepsilon\right). 
\end{equation}
%******************************************************************

%*******************************************************
\begin{remark}\label{rmk3.4}
Unless $r=1$ the roots $\alpha_j$ given by formula (\ref{eq3.10}) belong to the
distinct branches $W_j(z)$ of the Lambert $W$-function; hence the notation.   If $r=1$,  then $W_j(z)$
is the principal branch of the $W$-function. The value of $\log \varepsilon$ in (\ref{eq3.10}) is 
given by the complex logarithm $\log\rho +i(\theta+2\pi j)$, where $\varepsilon=\rho e^{2\pi i\theta}$. 
\end{remark}
%*******************************************************

\medskip 
(ii)   {\bf Case $\varepsilon\in\mathbf{R}$.}
 In this case $\varepsilon\in k\subset \mathbf{R}$
 and we use the second equation (\ref{eq3.2}). 
 Applying Euler's formula $\cos 2\pi\alpha_j=\frac{1}{2} \left(e^{2\pi i\alpha_j}+e^{-2\pi i\alpha_j}\right)$
 to the latter, one gets:
 %***********************************************
\begin{equation}\label{eq3.11}
2\alpha_j=\log\varepsilon ~e^{2\pi i\alpha_j}+\log\varepsilon ~e^{-2\pi i\alpha_j}. 
\end{equation}
%*****************************************************************

 Consider the following two equations:
 %***********************************************
\begin{equation}\label{eq3.12}
\left\{
\begin{array}{cl}
\alpha_j= & \log\varepsilon ~e^{2\pi i\alpha_j}, \\
\alpha_j= & \log\varepsilon ~e^{-2\pi i\alpha_j}.
\end{array}
\right.
\end{equation}
%*****************************************************************

Clearly, every pair of solutions $\alpha_j^{(1)}$ and  $\alpha_j^{(2)}$ of  equations (\ref{eq3.12})
imply a solution $\alpha_j$ of equation (\ref{eq3.11}) by the formula $\alpha_j= \alpha_j^{(1)}+\alpha_j^{(2)}$
and vice versa. Consider  Lemma \ref{lm2.1} for  the constants $A=0, B=\log\varepsilon$ and $C=\pm 2\pi i$. 
Since $BC\ne 0$ one can apply formulas (\ref{eq2.5}) to calculate the roots  $\alpha_j^{(1)}$ and  $\alpha_j^{(2)}$ of  equations (\ref{eq3.12}):
 %***********************************************
\begin{equation}\label{eq3.13}
\left\{
\begin{array}{ccc}
\alpha_j^{(1)} &= & -\frac{1}{2\pi i}W_j\left(-2\pi i ~\log\varepsilon\right), \\
\alpha_j^{(2)} &= & \frac{1}{2\pi i}W_j\left(2\pi i ~\log\varepsilon\right).
\end{array}
\right.
\end{equation}
%*****************************************************************

Therefore, one gets from (\ref{eq3.13}):
%***************************************************************
\begin{equation}\label{eq3.14}
\alpha_j= \alpha_j^{(1)}+\alpha_j^{(2)}=-\frac{1}{2\pi i}  ~\left[W_j(-2\pi i  ~\log \varepsilon)- W_j(2\pi i ~\log \varepsilon)\right].
\end{equation}
%*****************************************************************

\bigskip
Lemma \ref{lm3.3} is proved.
\end{proof}
%*********************************************

%***********************************************
\begin{corollary}\label{cor3.4}
The number fields  $K\cong k$ are isomorphic,  if and only if: 
%***********************************************
\begin{equation}\label{eq3.15}
\begin{cases} \alpha_j=-\frac{1}{2\pi i}  ~W_j(-2\pi\log \varepsilon), & if ~\varepsilon\in(\mathbf{C} - \mathbf{R})\cup\mathbf{Q},\cr
 \alpha_j=-\frac{1}{2\pi i}  ~\left[W_j(-2\pi i ~\log \varepsilon)- W_j(2\pi i ~\log \varepsilon)\right], & if ~\varepsilon\in\mathbf{R}.
                \end{cases}               
\end{equation}
%***************************************************
\end{corollary} 
%***********************************************
\begin{proof}
Corollary \ref{cor3.4} follows from Lemmas \ref{lm3.2} and \ref{lm3.3}.
\end{proof}
%*********************************************

\bigskip
Theorem \ref{thm1.1} follows from Corollary \ref{cor3.4}.

%**************************************************************************
\subsection{Proof of Corollary \ref{cor1.2}}
%***************************************************************************
For the sake of clarity, let us outline the main ideas. 
Consider the Hilbert class field $K$ of the number field $k$ and let $Cl~(k)$ be the class group of $k$.
 By the class field theory, the abelian groups
 $Cl~(k)\cong Gal~(K|k)\subseteq GL_r\left(A/aA\right)$  are trivial if and only if the class number $h_k:=|Cl~(k)|=1$. 
 In other words, the cardinality of the set $\{\alpha_j\}$ of solutions (\ref{eq1.3})  is equal to the 
 number of fields $k$, such that $h_k=1$.  On the other hand, the size of the set  $\{\alpha_j\}$ depends on how
 many distinct units $\varepsilon\in k$ the fields $k$  of fixed degree $2r$ over $\mathbf{Q}$ can afford. 
Dirichlet's Unit Theorem says that the number of such units is always infinite unless $k$ is an imaginary
quadratic field in which case there are only $8$ of them. 
Let us pass to a detailed argument.

\bigskip
Denote by $k$ a Galois extension of degree $2r$ over $\mathbf{Q}$. Recall that the group of units $O_k^{\times}$ of the 
field $k$ is described by  Dirichlet's Unit Theorem:
%********************************************************************
\begin{equation}\label{eq3.16}
O_k^{\times}\cong\mu(k)\oplus\mathbf{Z}^{\sigma_1+\sigma_2-1}, 
\end{equation}
%***********************************************************
where $\mu(k)$ is a finite group of the $n$-th roots of unity $\zeta_n$ contained in the field $k$
and $\sigma_1$ ($\sigma_2$, resp.) is the number of real (pairs of complex, resp.) 
embeddings of $k$, so that:
 %***********************************************
\begin{equation}\label{eq3.17}
\left\{
\begin{array}{ccc}
\sigma_1+2\sigma_2 &= & 2r, \\
\sigma_1\sigma_2 &= & 0.
\end{array}
\right.
\end{equation}
%*****************************************************************

The second line in  (\ref{eq3.17}) is true, since $k$ is a Galois extension, i.e. it is either totally real or totally imaginary extension 
of $\mathbf{Q}$.  

\bigskip
(i) {\bf Case $r= 1$  and $k\subset (\mathbf{C}-\mathbf{R})\cup\mathbf{Q}$.} 
In this case one gets $\sigma_1=0$ and $\sigma_2=1$
in formulas (\ref{eq3.17}), i.e. $k$ are imaginary quadratic fields. 
Hence by formula (\ref{eq3.16}) $O_k^{\times}\cong\mu(k)$ is a finite group. 
On the other hand, the cyclotomic quadratic number fields are well known
and their groups of units are exhausted by the following cases:
%*******************************************************************
\begin{equation}\label{eq3.18} 
\{\mathbf{Z}[\zeta_n]^{\times} ~|~n=1,2,3,4,6\}.  
\end{equation}
%*************************************************************  

It is immediate from (\ref{eq3.18}) that only $8$ units $\varepsilon\in k$ are
distinct, namely: 
 %*******************************************************************
\begin{equation}\label{eq3.19} 
\varepsilon\in \left\{1, ~\frac{1+i\sqrt{3}}{2}, ~i,  \frac{-1+i\sqrt{3}}{2}, ~-1, \frac{-1-i\sqrt{3}}{2}, ~-i, \frac{1-i\sqrt{3}}{2}\right\}.   
\end{equation}
%*************************************************************  

Likewise, $r=1$ implies $j=1$ in formulas (\ref{eq1.3}). We denote $\alpha=\alpha_1$ and 
let $W(z)=W_1(z)$ be the principal branch of the Lambert $W$-function. 
By Remark \ref{rmk3.4},  one gets $\log\varepsilon=i(\theta+2\pi)\ne 0$ in this case. 
Theorem \ref{thm1.1} says that cardinality of the set of number fields $\{k ~|~h_k=1\}$
is equal to such of the set of distinct roots: 
 %*******************************************************************
\begin{equation}\label{eq3.20} 
 \alpha=-\frac{1}{2\pi i}  ~W(-2\pi\log \varepsilon). 
  \end{equation}
%*************************************************************  

The  $W(z)$ is an invertible function, since its derivative $\frac{dW}{dz}=\frac{1}{z+e^{W(z)}}\ne 0$. 
We conclude that  distinct roots $\alpha$ given by (\ref{eq3.20}) 
are in a one-to-one correspondence  with the units $\varepsilon$ listed in  formula (\ref{eq3.19}).   
Therefore $\#\{k~|~h_k=1\}=8$ in this case, see also Remark \ref{rmk1.3}.

\medskip
(ii)  {\bf Case $r= 1$  and $k\subset \mathbf{R}$.} 
In this case  formulas (\ref{eq3.17}) imply  $\sigma_1=2$ and $\sigma_2=0$, 
i.e. $k$ are real  quadratic fields. 
Using Dirichlet's  formula (\ref{eq3.16}), one gets an infinite group of units:
%**********************************************************************
\begin{equation}\label{eq3.21}
O_k^{\times}\cong\mu(k) \oplus \mathbf{Z}. 
\end{equation}
%*******************************************************************

Again $r=1$ implies $j=1$ in formulas (\ref{eq1.3}). We denote $\alpha=\alpha_1$ and 
let $W(z)=W_1(z)$ be the principal branch of the Lambert $W$-function. 
Likewise,  Theorem \ref{thm1.1} implies that cardinality of the set of number fields $\{k ~|~h_k=1\}$
is equal to such of the set of distinct roots: 
 %*******************************************************************
\begin{equation}\label{eq3.22} 
 \alpha=-\frac{1}{2\pi i}  ~\left[W(-2\pi i\log \varepsilon)- W(2\pi i\log \varepsilon)\right].
 \end{equation}
%*************************************************************  

Function $W(-z)-W(z)$ at the RHS of (\ref{eq3.22}) is invertible, 
 since its derivative $\frac{1}{z-e^{W(-z)}}-\frac{1}{z+e^{W(z)}}\ne 0$;
 for otherwise one gets $W(-z)+\log (-1)=W(z)$ which is impossible due to
 formula (\ref{eq2.3}).  We conclude that  distinct roots $\alpha$ given by (\ref{eq3.22}) 
are in a one-to-one correspondence  with the units $\varepsilon\in O_k^{\times}$ given by Dirichlet's 
formula (\ref{eq3.21}).    Since there are infinitely many distinct values of $\varepsilon$,
we infer  from Theorem \ref{thm1.1}  that  $\#\{k~|~h_k=1\}=\infty$.

\medskip
(iii)   {\bf Case $r\ge  2$.}  It follows from Dirichlet's formula (\ref{eq3.16}) that the rank of the group 
of units of the field $k$ is defined as follows:
 %*******************************************************************
\begin{equation}\label{eq3.23} 
rank~O_k^{\times}=\sigma_1+\sigma_2-1.
 \end{equation}
%*************************************************************  

In view of formula (\ref{eq3.17}) we have the following two cases to consider:

\medskip
{\bf Case (a) $\sigma_1= 2r$ and $\sigma_2=0$.} In other words, the number fields $k$ are totally real of degree $\ge 4$ over $\mathbf{Q}$.   
By formula (\ref{eq3.23}) one gets $rank~O_k^{\times}=2r-1$.

\medskip
{\bf Case (b) $\sigma_1= 0$ and $\sigma_2=r$.} In this case the number fields $k$ are totally imaginary of degree $\ge 4$ over $\mathbf{Q}$.   
By formula (\ref{eq3.23}) we have $rank~O_k^{\times}=r-1$.

\bigskip
In  both of the above cases one gets $rank~O_k^{\times}\ge 1$. On the other hand, 
Theorem \ref{thm1.1} implies that cardinality of the set of number fields $\{k ~|~h_k=1\}$
is equal to such of the set of distinct roots: 
 %*******************************************************************
\begin{equation}\label{eq3.24} 
 \alpha_j=-\frac{1}{2\pi i}  ~\left[W_j(-2\pi i\log \varepsilon)- W_j(2\pi i\log \varepsilon)\right],  \quad 1\le j\le r.
 \end{equation}
%*************************************************************  

We repeat the argument in item (ii) and conclude that 
 distinct roots $\alpha_j$ given by (\ref{eq3.24}) 
are in a one-to-one correspondence  with the units $\varepsilon\in O_k^{\times}$.
It remains to notice that by $rank~O_k^{\times}\ge 1$ there are infinitely many  
distinct values of $\varepsilon$. 
We conclude  from Theorem \ref{thm1.1}  that  $\#\{k~|~h_k=1\}=\infty$ in this case.

 \bigskip
 Corollary \ref{cor1.2} is proved. 

\bigskip
%**************************************************
\begin{remark}
Roughly speaking, Corollary \ref{cor1.2} says that the Galois extensions $k$ of degree $2r$ over $\mathbf{Q}$ 
having class number one are classified by the units $\varepsilon\in O_{k'}^{\times}$ of a field $k'$ of the same 
degree over $\mathbf{Q}$. Notice that $k'\not\cong k$ due to a transcendental nature of the Lambert $W$-function
 [Corless, Gonnet, Hare, Jeffrey \&  Knuth 1996] \cite{CGHJK}.
Elsewhere, a relation between the units and class numbers is long known, e.g. for the cyclotomic fields 
[Sinnott 1980/81] \cite{Sin1}. 
\end{remark}
%*************************************************

  %%%%%%%%%%%%%%
\section*{Data availability}
%%%%%%%%%%%%%%%
  
  Data sharing not applicable to this article as no datasets were generated or analyzed during the current study.
   
%%%%%%%%%%%%%%
\section*{Conflict of interest}
%%%%%%%%%%%%%%%
On behalf of all co-authors, the corresponding author states that there is no conflict of interest.
  
  %%%%%%%%%%%%%%

%%%%%%%%%%%%%%
\section*{Funding declaration}
%%%%%%%%%%%%%%%
The author was partly supported by the NSF-CBMS grant 2430454.

\bibliographystyle{amsplain}

%**********************************************************

\end{document}